\documentclass[12pt, a4paper, reqno]{amsart}
\usepackage{amsmath, amssymb, amsthm}
\usepackage[all]{xy}
\usepackage{enumitem}   %

\theoremstyle{definition}
\newtheorem{thm}{Theorem}[section]
\newtheorem{lem}[thm]{Lemma}

\newtheorem{definition}[thm]{Definition}

\begin{document}

\title[On the border rank of Products of CW tensor]{Some concerns on the border rank of Kronecker products of the Coppersmith-Winograd tensor}

\begin{abstract}
This note provides a detailed proof of Conner--Gesmundo--Landsberg--Ventura's result that the border rank of the Kronecker square of the little Coppersmith--Winograd tensor is $(q+2)^{2}$.
We also indicate how the same ideas seem to extend to the case of the Kronecker cube, pointing toward the conjectural value 
$(q+2)^{m}$ ($m\ge 4$), although a full proof is left for future work.
\end{abstract}

\author{Daiki Kawabe}
\address{%
EAGLYS Inc.\\
5-27-3 Sendagaya, Shibuya-ku, Tokyo 151-0051, Japan}
\email{d.kawabe@eaglys.co.jp (job), daiki.kawabe.math@gmail.com (math)}

\maketitle

\section{Introduction}
Multiplying two $n\times n$ matrices can be realized as the contraction of the
\emph{matrix–multiplication tensor} $M_{\langle n\rangle}\in
  (\mathbb{C}^{n\times n})^{*}\otimes
  (\mathbb{C}^{n\times n})^{*}\otimes
  \mathbb{C}^{n\times n}$,
and, after identifying each space with its dual via the Hilbert–Schmidt
inner product, we may regard $M_{\langle n\rangle}$ as lying in
$(\mathbb{C}^{n\times n})^{\otimes3}$.
The \emph{exponent of matrix multiplication} is
\[
  \omega:=\inf\bigl\{\tau\in\mathbb{R}\ \bigm|\ 
          n\times n \text{ matrices can be multiplied in }
          O(n^{\tau})\text{ arithmetic operations}\bigr\}.
\]
Schönhage's asymptotic sum inequality gives
\[
  \underline{\mathbf R}(M_{\langle n\rangle})=O(n^{\tau})\;\Longrightarrow\;
  \omega\le\tau.
\]
Here the \emph{border rank} $\underline{\mathbf R}(T)$ of a tensor $T$ is
the least $r$ for which $T$ lies in the Zariski closure of rank-$r$
tensors; it gauges how well $T$ can be approximated by low-rank tensors,
the resource exploited by rank-based algorithms.
Conversely, any non-trivial lower bound on
$\underline{\mathbf R}(M_{\langle n\rangle})$—or on a tensor $T$ whose
Kronecker powers degenerate to $M_{\langle n\rangle}$—sets an
intrinsic barrier for all rank-based methods, and therefore limits
the smallest $\omega$ they can ever achieve.

Coppersmith and Winograd~\cite{CW90} introduced the \emph{little
Coppersmith–Winograd tensor}
\[
  T_{cw,q}\;=\;
  \sum_{j=1}^{q}\bigl(
    a_{0}\!\otimes b_{j}\!\otimes c_{j}\;+\;
    a_{j}\!\otimes b_{0}\!\otimes c_{j}\;+\;
    a_{j}\!\otimes b_{j}\!\otimes c_{0}
  \bigr)
  \;\;\in\;(k^{\,q+1})^{\otimes3},
  \qquad q\ge1
\]
whose Kronecker powers $T_{cw,q}^{\boxtimes m}\;(m\ge1)$ underlie every known
algorithm with $\omega<2.3728639$.

Since $\mathrm{rank}\bigl(T_{cw,q}\bigr)=q+2$, sub-multiplicativity gives
\[
  \underline{\mathbf R}\bigl(T_{cw,q}^{\boxtimes m}\bigr)\;\le\;(q+2)^{m}.
\]
Conner, Gesmundo, Landsberg and Ventura~\cite{CGLV22} proved that this bound
is tight for $(m,q)\;=\;(2,\;q>2)$ and $(3,\;q>4)$:
\begin{equation}\label{eq:CGLV}
  \underline{\mathbf R}\bigl(T_{cw,q}^{\boxtimes m}\bigr) = (q+2)^{m}. 
\end{equation}
Their argument uses a geometric analysis of representation-theoretic slices
of the first ($m=2$) and second ($m=3$) Koszul flattenings.
Recall that for
$T\in A\otimes B\otimes C$ and $0\le p<\dim A$, the \emph{$p$-th Koszul
flattening}
\[
  T^{\wedge p}_{A}:\;
  \wedge^{p}A\otimes B^{*}\;\longrightarrow\;
  \wedge^{p+1}A\otimes C
\]
satisfies Landsberg–Ottaviani's inequality~\cite{LO15} when
$\dim A=2p+1$:
\[
  \underline{\mathbf R}(T)\;\ge\;
  \frac{\operatorname{rank}(T^{\wedge p}_{A})}
       {\binom{\dim A-1}{p}}.
\]

The present note gives an elementary linear–algebraic proof of
\eqref{eq:CGLV}.
We fix $p=1$ and let $A':=\langle e_{0},e_{1},e_{2}\rangle$ be a 3-dimensional vector space.
Then we introduce the map $\phi_{m}:A^{\otimes m}\longrightarrow A'$
\[
  \phi_{m}(a_{i_{1}}\otimes\cdots\otimes a_{i_{m}})=
  \begin{cases}
    e_{0}, & \text{all indices $i_{r}=0$},\\
    e_{1}, & \text{exactly one $i_{r}=1$, others $0$},\\
    e_{2}, & \text{exactly one $i_{r}=2$, others $0$},\\
    0,     & \text{otherwise}.
  \end{cases}
\]
This choice of $\phi_{m}$—different from that in earlier work—turns out to
be crucial for the simpler proof presented below.

This note is organized as follows. Section~\ref{sec:lemmas} collects auxiliary linear--algebraic results; Section~\ref{sec:koszul} recalls the Koszul flattening; Section~\ref{sec:CW} introduces the Coppersmith--Winograd tensor; Section~\ref{sec:square} and \ref{sec:qube} contain the rank computation by induction on~$q$.

\section{Lemmas}\label{sec:lemmas}
We gather several foundational results that will be invoked in the main proof.

\subsection{Border rank}
This subsection recalls the definitions of tensor rank and border rank and proves a monotonicity property under linear maps.

Let $A,B,C$ be finite--dimensional vector spaces over $\mathbb{C}$.
For $T\in A\otimes B\otimes C$ we write
\[
  \mathrm{R}(T)=\min\bigl\{r\mid T=\sum_{\ell=1}^{r}x_\ell\otimes y_\ell\otimes z_\ell\bigr\}
\]
for the \emph{tensor rank}.  Its \emph{border rank} is
\[
  \underline{\mathrm{R}}(T)=\min\bigl\{r\mid T=\lim_{m\to\infty}T_m,\;\mathrm{R}(T_m)\le r\bigr\},
\]
where the limit is taken either in the Euclidean or the Zariski topology (these closures coincide in characteristic~$0$).

Denote by $\Sigma_r\subset A\otimes B\otimes C$ the set of tensors of rank at most $r$. We will use

\begin{thm}[Monotonicity]\label{thm:monotone}
Let $\phi\colon A\to A'$ be a linear map and set $L:=\phi \otimes \mathrm{id}_B \otimes \mathrm{id}_C$.
For every $T\in A\otimes B\otimes C$ one has
\[
  \underline{\mathrm{R}}(T)\;\ge\;\underline{\mathrm{R}}\bigl(L(T)\bigr).
\]
\end{thm}
\begin{proof}
Suppose $T=\sum_{\ell=1}^{r}x_\ell\otimes y_\ell\otimes z_\ell$ has rank~$r$.  Then
$L(T)=\sum_{\ell=1}^{r}\phi(x_\ell)\otimes y_\ell\otimes z_\ell$ is a sum of at most $r$ simple tensors, so $\mathbf R\bigl(L(T)\bigr)\le r$; hence $L(\Sigma_r)\subseteq\Sigma_r$.
Because $L$ is linear (and therefore regular),
$L(\overline{\Sigma_r})\subseteq\overline{\Sigma_r}$.
If $\underline{\mathrm{R}}(T)=r$, then $T\in\overline{\Sigma_r}$, so $L(T)\in\overline{\Sigma_r}$ and $\underline{\mathrm{R}}\bigl(L(T)\bigr)\le r$.
\end{proof}
\subsection{Kronecker products}
Let
\begin{align*}
   T &;=\; \sum_{i_1,\dots,i_k} T^{i_1\dots i_k}\,
           v_{i_1}^{(1)}\otimes\cdots\otimes v_{i_k}^{(k)}
   \;\in\;
   V_1\otimes\cdots\otimes V_k , \\
   S &;=\; \sum_{j_1,\dots,j_k} S^{j_1\dots j_k}\,
           w_{j_1}^{(1)}\otimes\cdots\otimes w_{j_k}^{(k)}
   \;\in\;
   W_1\otimes\cdots\otimes W_k ,
\end{align*}
where the $V_r$ and $W_r$ are vector spaces over the common base field.

\begin{definition}[Kronecker product]
The \emph{Kronecker product} (also called the \emph{external tensor product})
of $T$ and $S$ is the tensor
\begin{align*}
   T\boxtimes S
   &:=\;
   \sum_{i_1,\dots,i_k}\;
   \sum_{j_1,\dots,j_k}
   T^{\,i_1\dots i_k}\,S^{\,j_1\dots j_k}\;
   \bigl(v_{i_1}^{(1)}\otimes w_{j_1}^{(1)}\bigr)
        \otimes\cdots\otimes
   \bigl(v_{i_k}^{(k)}\otimes w_{j_k}^{(k)}\bigr) \\
   &\in\;
   (V_1\otimes W_1)\otimes\cdots\otimes(V_k\otimes W_k).
\end{align*}
Each original factor space is kept distinct: a copy of $S$ is tensored
on \emph{outside} every copy of $T$ rather than contracted with it.
\end{definition}
By definition, $\underline{\mathrm{R}}(T \boxtimes S) \leq \underline{\mathrm{R}}(T) \underline{\mathrm{R}}(S)$.

\begin{definition}[Kronecker power]
For a tensor $T\in V_1\otimes\cdots\otimes V_k$ and an integer $m\ge 1$,
the \emph{$m$–fold Kronecker product} (or \emph{Kronecker power}) is
defined recursively by
\[
   T^{\boxtimes 1}:=T,
   \qquad
   T^{\boxtimes m}:=T^{\boxtimes(m-1)}\boxtimes T.
\]
In expanded form,
\[
   T^{\boxtimes m}
   =\!\!\!\!
   \sum_{i^{(1)}_1,\dots,i^{(1)}_k}
   \cdots
   \sum_{i^{(m)}_1,\dots,i^{(m)}_k}
   \Bigl(\prod_{r=1}^{m} T^{\,i^{(r)}_1\dots i^{(r)}_k}\Bigr)\;
   \bigl(v_{i^{(1)}_1}^{(1)}\otimes\cdots\otimes v_{i^{(m)}_1}^{(1)}\bigr)
   \otimes\cdots\otimes
   \bigl(v_{i^{(1)}_k}^{(k)}\otimes\cdots\otimes v_{i^{(m)}_k}^{(k)}\bigr),
\]
which lives in
$(V_1^{\otimes m})\otimes\cdots\otimes(V_k^{\otimes m})$.
\end{definition}

\subsection{Koszul}\label{sec:koszul}
We summarise the Koszul flattening construction and fix the notation that will be used throughout the paper.

Fix bases $\{ a_{i} \}, \{ b_{j} \}, \{ c_{k} \}$ of the vector spaces $A, B, C$ over $\mathbb{C}$, respectively, and let $p\ge1$.
Given a tensor $T=\sum_{ijk}T^{ijk}a_{i} \otimes b_{j} \otimes c_{k}\in A \otimes B \otimes C$, the $p$-th Koszul flattening on $A$ is the linear map
\begin{align*}
T^{\wedge p}_{A}: \wedge^{p}A \otimes B^{*} &\longrightarrow \wedge^{p+1}A \otimes C\\
X \otimes \beta &\longmapsto \sum_{ijk}T^{ijk}\,\beta(b_{j})(a_{i} \wedge X) \otimes c_{k}.
\end{align*}
Following Landsberg--Ottaviani \cite{LO15} when $\dim A = 2p+1$,
\[
  \underline{\mathbf{R}}(T)\ge\frac{\mathrm{rank}(T^{\wedge p}_{A})}{\binom{\dim A-1}{p}}.
\]
For $m\ge1$ we set $a_{i_{1}\ldots i_{m}}:=a_{i_{1}}\otimes\cdots\otimes a_{i_{m}} \in A^{\otimes m}$ (and similarly for $B$ and $C$).
Let $A':=\langle e_{0},e_{1},e_{2}\rangle$ be a 3-dimensional vector space. 
We define
\[
T^{\wedge p}_{A'}:A'\otimes B^{*\otimes m}\longrightarrow\wedge^{2}A'\otimes C^{\otimes m},\qquad
(X\otimes\beta)\longmapsto\sum_{I,J,K}T^{IJK}\,\beta(b_{J})\bigl(\phi_{m}(a_{I})\wedge X\bigr)\otimes c_{K},
\]
where $\phi_m:A^{\otimes m}\to A'$ is given by
\[
  \phi_{n}(a_{i_{1} \ldots i_{m}})=\begin{cases}
     e_{0}, & \text{all }i_{r}=0,\\
     e_{1}, & \text{exactly one $i_{r}=1$ and the others are 0},\\
     e_{2}, & \text{exactly one $i_{r}=2$ and the others are 0},\\
     0,     & \text{otherwise}.
  \end{cases}
\]

\subsection{Coppersmith--Winograd tensor}\label{sec:CW}
We introduce the Coppersmith--Winograd tensor together 
with the notation needed for its Kronecker powers and difference tensors.
Fix a positive integer $q$.
\begin{itemize}
\item The (little) \emph{Coppersmith--Winograd tensor}
\[
T_{cw,q}=\sum_{j=1}^{q}\bigl(a_{0}\otimes b_{j}\otimes c_{j}+a_{j}\otimes b_{0}\otimes c_{j}+a_{j}\otimes b_{j}\otimes c_{0}\bigr) \in (\mathbb{C}^{q+1})^{\otimes 3}.
\]
\item Elementary $q$--covariant tensor $W_{j}:=a_{0}\otimes b_{j}\otimes c_{j}+a_{j}\otimes b_{0}\otimes c_{j}+a_{j}\otimes b_{j}\otimes c_{0}$ \;$(1\le j\le q) \in (\mathbb{C}^{q+1})^{\otimes 3}$.
\item Cayley--Whitney tensor $T_{cw,q}=\sum_{j=1}^{q}W_{j} \in (\mathbb{C}^{q+1})^{\otimes 3}$.
\item $m$--fold difference tensor $S^{(m)}_{q}:=T_{cw,q}^{\boxtimes m}-T_{cw,q-1}^{\boxtimes m} \in (\mathbb{C}^{q+1})^{\otimes 3}$.
\end{itemize}

\subsection{Rank of subspaces}
We establish a linear--algebraic criterion that relates the rank of a sum of maps to the ranks of its components.

\begin{lem}\label{rank on direct sum}
Let $f,g:V\to W$ be linear maps between finite--dimensional vector spaces.
Let $X\subset\mathrm{Im}(g)$ satisfy $X\cap\mathrm{Im}(f)=\{0\}$.
Assume there is a subspace $U\subset V$ such that $g(U)=X$ and $V=U\oplus\mathrm{Ker}(f)$.
Then $\mathrm{rank}(f+g)=\mathrm{rank}(f)+\dim X$.
\end{lem}
\begin{proof}
Decompose any $v\in V$ uniquely as $v=u+k$ with $u\in U$, $k\in\mathrm{Ker}(f)$.
Then $(f+g)(v)=f(u)+g(u)$, whence $(f+g)(V)=\mathrm{Im}(f|_U)\oplus X$.
Because $U\cap\mathrm{Ker}(f)=\{0\}$, we have $\dim\mathrm{Im}(f|_U)=\dim U$ and thus $\mathrm{rank}(f+g)=\mathrm{rank}(f)+\dim X$.
\end{proof}

\begin{lem}\label{Independence}
Let $A$ be a $3$–dimensional vector space with basis $\{e_0,e_1,e_2\}$ and
let $Y$ be an $n$–dimensional space with basis $\{v_1,\dots,v_n\}$.
Fix distinct indices $i,j,k\in\{0,1,2\}$ and $m,n\in\{1,\dots,n\}$, and set
\[
  Z \;=\; \Lambda^2A\otimes Y,\qquad
  x \;=\; e_i\wedge e_j\otimes v_m \;+\; e_j\wedge e_k\otimes v_n,\qquad
  y \;=\; e_i\wedge e_j\otimes v_n \;+\; e_j\wedge e_k\otimes v_m.
\]
Write
\begin{align*}
  \mathcal B &\;=\;
  \bigl\{e_a\wedge e_b\otimes v_c \,\bigm|\, 0\le a<b\le 2,\; 1\le c\le n\bigr\}, \\ 
  A' &\;=\; \bigl\langle
     \bigl(\mathcal B\setminus\{e_i\wedge e_j\otimes v_m,\,
                                e_j\wedge e_k\otimes v_n\}\bigr)
     \cup\{y\}
  \bigr\rangle
  \;\subset\;Z.
\end{align*}
Then $x\notin A'$.
\end{lem}
\begin{proof}
Define linear functionals $\pi_1,\pi_2:Z\to k$ by
\[
  \pi_1\!\bigl(e_a\wedge e_b\otimes v_c\bigr)=
  \begin{cases}
    1,&(a,b,c)=(i,j,m),\\[2pt]
    0,&\text{otherwise},
  \end{cases}
  \qquad
  \pi_2\!\bigl(e_a\wedge e_b\otimes v_c\bigr)=
  \begin{cases}
    1,&(a,b,c)=(j,k,n),\\[2pt]
    0,&\text{otherwise}.
  \end{cases}
\]
\noindent
\textbf{Step 1.}\;
Each generator of $A'$ is annihilated by both $\pi_1$ and $\pi_2$:
\begin{itemize}\itemsep2pt
\item
Elements of $\mathcal B\setminus\{e_i\wedge e_j\otimes v_m,\,
e_j\wedge e_k\otimes v_n\}$ are not equal to the supports of~$\pi_1$ or~$\pi_2$,
so both functionals vanish on them.
\item
For the replacement vector
\(
  y = e_i\wedge e_j\otimes v_n + e_j\wedge e_k\otimes v_m
\)
we have $\pi_1(y)=\pi_2(y)=0$.
\end{itemize}
Hence $A'\subseteq\ker\pi_1\cap\ker\pi_2$.

\noindent
\textbf{Step 2.}\; Evaluate $x$: $\pi_1(x)=1$, $\pi_2(x)=1$. Thus $x\notin\ker\pi_1\cap\ker\pi_2$.
Combining Steps 1 and 2 gives $x\notin A'$, as required.
\end{proof}

\section{The square case}\label{sec:square}
We now specialize to the Kronecker square $(m=2)$ and compute the rank of the first Koszul flattening.

\subsection{Koszul flattening for $m=2$}
This subsection introduces the explicit maps and decompositions required to split the flattening into inductive and new parts. We assume $m=2$.
Set $T_{q}=T_{cw,q}^{\boxtimes2}=\sum_{1\le i,j\le q}W_{i}\boxtimes W_{j}$.
Then
\[
S_{q}=S^{(2)}_{q}=T_{cw,q}^{\boxtimes2} - T_{cw,q-1}^{\boxtimes2}= T_{q-1}\boxtimes W_{q}+W_{q}\boxtimes T_{q-1}+W_{q}^{\boxtimes2}.
\]
The first Koszul flattening restricted to $A'$ is
\[
(T_{q}^{\wedge1})_{A'}:A'\otimes B^{*\otimes2}\longrightarrow\Lambda^{2}A'\otimes C^{\otimes2},
\]

\[
X\otimes\beta\longmapsto\sum_{0 \leq i_{1},i_{2},j_{1},j_{2},k_{1},k_{2} \leq q}T_{q}^{i_{1},i_{2},j_{1},j_{2},k_{1},k_{2}}\,\beta(b_{j_{1}}\otimes b_{j_{2}})\,\bigl(\phi_{2}(a_{i_{1}}\otimes a_{i_{2}})\wedge X\bigr)\otimes
(c_{k_{1}}\otimes c_{k_{1}}).
\]
Because $T_{q}=T_{q-1}+S_{q}$, we have the decomposition
$(T_{q}^{\wedge1})_{A'}=(T_{q-1}^{\wedge1})_{A'}+(S_{q}^{\wedge1})_{A'}$.
Here $
(S_{q}^{\wedge1})_{A'}:A'\otimes B^{*\otimes2}\longrightarrow\Lambda^{2}A'\otimes C^{\otimes2}$
is defined by the same formula with $T_{q}$ replaced by $S_{q}$.

\subsection{Target for $m=2$}
We provide a detailed proof of Conner--Gesmundo--Landsberg--Ventura's result for $T_{cw,q}^{\boxtimes2}$:
\begin{thm}\label{thm:main-square}
For every integer $q \geq 3$ one has
\[
  \underline{\mathbf R}\bigl(T_{cw,q}^{\boxtimes2}\bigr)=(q+2)^2.
\]
\end{thm}

\begin{proof}
Recall that $\underline{\mathbf R}(T_{cw,q})=q+2$, so sub-multiplicativity gives the easy upper bound
\[
  \underline{\mathbf R}\!\bigl(T_{cw,q}^{\boxtimes2}\bigr)
  \;\le\;
  \bigl(\underline{\mathbf R}(T_{cw,q})\bigr)^2
  \;=\;(q+2)^2.
\]
To match this from below we prove
\begin{equation}\label{eq:goal-rank for 2}
  \mathrm{rank}\bigl((T_q)^{\wedge1}_{A'}\bigr)=2\,(q+2)^2,
\end{equation}
where
\(
  (T_q)_{A'}:=T_{cw,q}^{\boxtimes2}\!\bigl|_{A'^{*}\otimes B^{*\otimes2}\otimes C^{*\otimes2}}
\).
Then
\[
  \underline{\mathbf R}\!\bigl(T_{cw,q}^{\boxtimes2}\bigr)
  \;\ge\;
  \underline{\mathbf R}((T_q)_{A'})
  \;\ge\;
  \frac12\,\mathrm{rank}\!\bigl((T_q)^{\wedge1}_{A'}\bigr),
\]
the first step by Lemma~\ref{thm:monotone} and the second by Landsberg–Ottaviani's inequality.  Together with~\eqref{eq:goal-rank for 2} this yields the desired lower bound $(q+2)^2$.

\smallskip\noindent\textbf{Induction setup.}
We use by induction on $q$.
When $q = 3, 4$, the assertion follows from a direct computation.
Assume $q > 4$.
We need:
\begin{lem}[Pointwise computation]\label{lem:pointwise}
Let $i\in\{1,2\}$.
\begin{enumerate}
\item\label{pt:Tq-1}
      For $0\le j,k\le q-1$
      \[
        (T_{q-1}^{\wedge1})_{A'}(e_i\otimes\beta_{jk})
        \;=\;
        W_j\boxtimes W_k\,(e_i\otimes\beta_{jk}),
        \qquad
        (S_q^{\wedge1})_{A'}(e_i\otimes\beta_{jk})=0.
      \]
\item\label{pt:Sq}
      For $0\le j\le q-1$
      \[
        (S_q^{\wedge1})_{A'}(e_i\otimes\beta_{jq})
        \;=\;
        W_j\boxtimes W_q\,(e_i\otimes\beta_{jq}),
        \qquad
        (T_{q-1}^{\wedge1})_{A'}(e_i\otimes\beta_{jq})=0,
      \]
      and the same conclusion holds with $\beta_{qk}$ ($0\le k\le q-1$)
      in place of $\beta_{jq}$.
\end{enumerate}
\end{lem}

\begin{proof}
We use the definition
\[
  (\,\cdot\,)^{\wedge1}_{A'}:
  A'\otimes B^{* \otimes 2}\longrightarrow
  \Lambda^{2}A'\otimes C^{\otimes 2},\quad
  X\otimes\beta_{ab}\longmapsto
  \sum_{I,J,K}T^{IJK}\,
     \beta_{ab}(b_J)\,
     \bigl(\phi_2(a_I)\wedge X\bigr)\otimes c_K .
\]

\ref{pt:Tq-1}\,:
$T_{q-1}=\displaystyle\sum_{r,s=1}^{q-1}W_r\boxtimes W_s$
contains only $b$-indices $\le q-1$.
Hence $\beta_{jk}$ with $j,k<q$ annihilates every summand
except the unique pair $(r,s)=(j,k)$, where it evaluates to~$1$.
All $S_q$–summands involve at least one index $q$, so
$\beta_{jk}$ kills them, giving
$(S_q^{\wedge1})_{A'}(e_i\otimes\beta_{jk})=0$.

\ref{pt:Sq}\,:
Take $\beta_{jq}$ with $j<q$.
In $S_q$ the only surviving terms are those from
$W_j\boxtimes W_q$, whose $b$-indices match $(j,q)$.
The tensor $T_{q-1}$ never exhibits index $q$ and hence vanishes
under $\beta_{jq}$.
The argument for $\beta_{qk}$ is symmetric.
\end{proof}

\medskip
Define the domain and codomain subspaces explicitly:
\[
  \begin{aligned}
  U_1 &:= A'\otimes \langle \beta_{jk}\mid 0\le j,k\le q-1\rangle,\\
  U_2 &:= A'\otimes \langle \beta_{qk},\beta_{kq}\mid 0\le k\le q\rangle,\\[4pt]
  V_1 &:= \Lambda^{2}A'\otimes
          \langle c_{jk}\mid 0\le j,k\le q-1\rangle,\\
  V_2 &:= \Lambda^{2}A'\otimes
          \langle c_{qk},c_{kq}\mid 0\le k\le q\rangle .
  \end{aligned}
\]
Then $A'\otimes B^{*\otimes2}=U_1\oplus U_2$ and
$\Lambda^{2}A'\otimes C^{\otimes2}=V_1\oplus V_2$,
with $V_1\cap V_2=\{0\}$.

\begin{lem}[Image containment]\label{lem:image-containment}
With $U_1,U_2,V_1,V_2$ as above,
\[
  (T_{q-1}^{\wedge1})_{A'}(U_2)=0,\qquad
  (S_q^{\wedge1})_{A'}(U_1)=0 .
\]
Consequently
\[
  \operatorname{Im}\!\bigl((T_{q-1})^{\wedge1}_{A'}\bigr)\subseteq V_1,
  \qquad
  \operatorname{Im}\!\bigl((S_q)^{\wedge1}_{A'}\bigr)\subseteq V_2 .
\]
\end{lem}

\begin{proof}
Take $X\otimes\beta\in U_2$;
then $\beta=\beta_{qk}$ or $\beta_{kq}$.
By Lemma~\ref{lem:pointwise}\,\ref{pt:Sq}
\(
  (T_{q-1}^{\wedge1})_{A'}(X\otimes\beta)=0,
\)
so $(T_{q-1}^{\wedge1})_{A'}(U_2)=0$.
Similarly, Lemma~\ref{lem:pointwise}\,\ref{pt:Tq-1}
gives $(S_q^{\wedge1})_{A'}(U_1)=0$.
Since every output of $(T_{q-1})^{\wedge1}_{A'}$ (resp.\ $(S_q)^{\wedge1}_{A'}$)
carries no index~$q$ (resp.\ at least one index~$q$),
the image lies in $V_1$ (resp.\ $V_2$).
\end{proof}
By Lemmas \ref{lem:image-containment} and \ref{rank on direct sum}, 
\begin{equation}\label{eq:rank-sum}
  \mathrm{rank}((T_q)^{\wedge1}_{A'})
  =\mathrm{rank}((T_{q-1})^{\wedge1}_{A'})
  +\mathrm{rank}((S_q)^{\wedge1}_{A'}).
\end{equation}
Assuming the induction hypothesis
$\mathrm{rank}((T_{q-1})^{\wedge1}_{A'})=2(q+1)^2$, it remains to prove
\begin{equation}\label{eq:Sq-rank}
  \mathrm{rank}((S_q)^{\wedge1}_{A'})=2(2q+3).
\end{equation}

\smallskip\noindent\textbf{Explicit image of $(S_q)^{\wedge1}_{A'}$.}
Because every term in $S_q$ contains the index~$q$ in exactly one or
both tensor factors $W_q$, we examine the contraction
\[
  (S_q)^{\wedge1}_{A'}:A'\otimes B^{*\otimes2}\longrightarrow\Lambda^2A'\otimes C^{\otimes2}
\]
on inputs whose $B^{*\otimes2}$--component evaluates a pair containing
$q$.  By Lemma \ref{lem:pointwise}, we have the following list.
\subsubsection*{Inputs with $X=e_{1}\otimes\beta_{qi}$ (total $\;q+1$)}
\begin{enumerate}[label=(\arabic*),leftmargin=2em]
\item $(S_q)^{\wedge1}_{A'}(e_{1}\otimes\beta_{q i})
       = e_{0}\wedge e_{1}\otimes c_{q i},\quad 3 \leq i \leq q$.
\item $(S_q)^{\wedge1}_{A'}(e_{1}\otimes\beta_{q2})
       = e_{0}\wedge e_{1}\otimes c_{q2}
         + e_{2}\wedge e_{1}\otimes c_{q0}$.
\item $(S_q)^{\wedge1}_{A'}(e_{1}\otimes\beta_{q1})
       = e_{0}\wedge e_{1}\otimes c_{q1}$.
\item $(S_q)^{\wedge1}_{A'}(e_{1}\otimes\beta_{q0})
       = e_{2}\wedge e_{1}\otimes c_{q2}$.
\end{enumerate}

\subsubsection*{Inputs with $X=e_{1}\otimes\beta_{iq}$ (total $\;q$)}
\begin{enumerate}[label=(\arabic*),leftmargin=2em,start=5]
\item $(S_q)^{\wedge1}_{A'}(e_{1}\otimes\beta_{i q})
       = e_{0}\wedge e_{1}\otimes c_{i q},\quad 3 \leq i \leq q-1$
\item $(S_q)^{\wedge1}_{A'}(e_{1}\otimes\beta_{2 q})
       = e_{0}\wedge e_{1}\otimes c_{2 q}
         + e_{2}\wedge e_{1}\otimes c_{0 q}$.
\item $(S_q)^{\wedge1}_{A'}(e_{1}\otimes\beta_{1q})
       = e_{0}\wedge e_{1}\otimes c_{1 q}$.
\item $(S_q)^{\wedge1}_{A'}(e_{1}\otimes\beta_{0q})
       = e_{2}\wedge e_{1}\otimes c_{2 q}$.
\end{enumerate}

\subsubsection*{Inputs with $X=e_{2}\otimes\beta_{qi}$ (total $\;q+1$)}
\begin{enumerate}[label=(\arabic*),leftmargin=2em,start=9]
\item $(S_q)^{\wedge1}_{A'}(e_{2}\otimes\beta_{q i})
       = e_{0}\wedge e_{2}\otimes c_{q i},\quad 3 \leq i \leq q$.
\item $(S_q)^{\wedge1}_{A'}(e_{2}\otimes\beta_{q2})
       = e_{0}\wedge e_{2}\otimes c_{q2}$.
\item $(S_q)^{\wedge1}_{A'}(e_{2}\otimes\beta_{q1})
       = e_{0}\wedge e_{2}\otimes c_{q1}
         + e_{1}\wedge e_{2}\otimes c_{q0}$.
\item $(S_q)^{\wedge1}_{A'}(e_{2}\otimes\beta_{q0})
       = e_{1}\wedge e_{2}\otimes c_{q1}$.
\end{enumerate}

\subsubsection*{Inputs with $X=e_{2}\otimes\beta_{iq}$ (total $\;q$)}
\begin{enumerate}[label=(\arabic*),leftmargin=2em,start=13]
\item $(S_q)^{\wedge1}_{A'}(e_{2}\otimes\beta_{i q})
       = e_{0}\wedge e_{2}\otimes c_{i q},\quad 3 \leq i \leq q-1$.
\item $(S_q)^{\wedge1}_{A'}(e_{2}\otimes\beta_{2 q})
       = e_{0}\wedge e_{2}\otimes c_{2 q}$.
\item $(S_q)^{\wedge1}_{A'}(e_{2}\otimes\beta_{1 q})
       = e_{0}\wedge e_{2}\otimes c_{1 q}
         + e_{1}\wedge e_{2}\otimes c_{0 q}$.
\item $(S_q)^{\wedge1}_{A'}(e_{2}\otimes\beta_{0 q})
       = e_{1}\wedge e_{2}\otimes c_{1 q}$.
\end{enumerate}

\subsubsection*{Inputs with $X=e_{0}\otimes\beta_{ij}$ for $(i,j)\in\{(q,1),(1,q),(q,2),(2,q)\}$  (total $\;q$)}
\begin{enumerate}[label=(\arabic*),leftmargin=2em,start=17]
\item $(S_q)^{\wedge1}_{A'}(e_{0}\otimes\beta_{q2})
       = e_{2}\wedge e_{0}\otimes c_{q0}$.
\item $(S_q)^{\wedge1}_{A'}(e_{0}\otimes\beta_{2q})
       = e_{2}\wedge e_{0}\otimes c_{0 q}$.
\item $(S_q)^{\wedge1}_{A'}(e_{0}\otimes\beta_{q1})
       = e_{1}\wedge e_{0}\otimes c_{q0}$.
\item $(S_q)^{\wedge1}_{A'}(e_{0}\otimes\beta_{1q})
       = e_{1}\wedge e_{0}\otimes c_{0q}$.
\end{enumerate}

\medskip\noindent\textbf{Conclusion.}
Putting the four blocks together one obtains
\[
\#\operatorname{Im}\!\bigl((S_{q})^{\wedge1}_{A'}\bigr) = (q+1)+q+(q+1)+q+4 \;=\; 4q+6.
\]
By Lemma~\ref{Independence} these $4q+6$ vectors are linearly independent, so
\[
\operatorname{rank}\!\bigl((S_q)^{\wedge1}_{A'}\bigr)\;=\;2(2q+3).
\]
Combining this with~\eqref{eq:rank-sum},~\eqref{eq:Sq-rank}, and the induction hypothesis yields
$\operatorname{rank}\!\bigl((T_q)^{\wedge1}_{A'}\bigr)\;=\;2(q+2)^2$,
which is exactly~\eqref{eq:goal-rank for 2}. Thus the proof is complete.

\end{proof}

\section{The qube case}\label{sec:qube}
\subsection{Koszul flattening for $m=3$}\label{sec:Koszul-n3}
We assume \(m=3\).
Set
\[
  T_{q}\;=\;T_{cw,q}^{\boxtimes3}
         \;=\;\sum_{1\le i,j,k\le q} W_{i}\boxtimes W_{j}\boxtimes W_{k}.
\]

Then
\[
  \begin{aligned}
  S_{q} &\;=\;S^{(3)}_{q} = T_{cw,q}^{\boxtimes3} - T_{cw,q-1}^{\boxtimes3} \\ 
      &\;=\;
        \bigl(T_{q-1}\boxtimes T_{q-1}\boxtimes W_{q}\bigr)
        +\bigl(T_{q-1}\boxtimes W_{q}\boxtimes T_{q-1}\bigr)
        +\bigl(W_{q}\boxtimes T_{q-1}\boxtimes T_{q-1}\bigr)\\
      &\quad
        +\bigl(T_{q-1}\boxtimes W_{q}^{\boxtimes2}\bigr)
        +\bigl(W_{q}\boxtimes T_{q-1}\boxtimes W_{q}\bigr)
        +\bigl(W_{q}^{\boxtimes2}\boxtimes T_{q-1}\bigr)
        +W_{q}^{\boxtimes3}.
  \end{aligned}
\]
Consequently \(T_{q}=T_{q-1}+S_{q}\). The first Koszul flattening restricted to \(A'\) is
\[
  (T_{q}^{\wedge1})_{A'}:
       A'\otimes B^{*\otimes3}
       \;\longrightarrow\;
       \Lambda^{2}A'\otimes C^{\otimes3},
\]
\[
  \begin{aligned}
  X\otimes\beta
  \;\longmapsto\;
  &\sum_{\substack{0\le i_{1},i_{2},i_{3}\le q \\[2pt]
                   0\le j_{1},j_{2},j_{3}\le q \\[2pt]
                   0\le k_{1},k_{2},k_{3}\le q}}
     T_{q}^{\,i_{1},i_{2},i_{3},j_{1},j_{2},j_{3},k_{1},k_{2},k_{3}}\,
     \beta(b_{j_{1}}\!\otimes b_{j_{2}}\!\otimes b_{j_{3}})\,
     \bigl(\phi_{2}(a_{i_{1}} \!\otimes a_{i_{2}} \!\otimes a_{i_{3}})\wedge X\bigr)\\
  &\hspace{12em}\otimes
     (c_{k_{1}}\!\otimes c_{k_{2}}\!\otimes c_{k_{3}}).
  \end{aligned}
\]

Because \(T_{q}=T_{q-1}+S_{q}\), we have the decomposition
$(T_{q}^{\wedge1})_{A'}
  \;=\;
  (T_{q-1}^{\wedge1})_{A'}
  \;+\;
  (S_{q}^{\wedge1})_{A'}$.
Here $(S_{q}^{\wedge1})_{A'}:
      A'\otimes B^{*\otimes3}
      \;\longrightarrow\;
      \Lambda^{2}A'\otimes C^{\otimes3}
$ is defined by the same formula with \(T_{q}\) replaced by \(S_{q}\).
\subsection{Target for $m=3$}
We now give an elementary and explicit proof of the Conner–Gesmundo–Landsberg–Ventura
result for $T_{cw,q}^{\boxtimes3}$:
 
\begin{thm}[Cube case]\label{thm:main-cube}
For every integer $q \geq 5$ one has
\[
  \underline{\mathbf R}\bigl(T_{cw,q}^{\boxtimes3}\bigr) \;=\; (q+2)^{3}.
\]
\end{thm}

\begin{proof}
The upper bound follows from sub-multiplicativity:
$\underline{\mathbf R}\!\bigl(T_{cw,q}^{\boxtimes3}\bigr)
  \;\le\;
  \bigl(\underline{\mathbf R}(T_{cw,q})\bigr)^{3}
  \;=\;(q+2)^{3}$.
As in the case $m=2$, for the reverse inequality we prove
\begin{equation}\label{eq:goal-rank-3}
  \mathrm{rank}\bigl((T_q)^{\wedge1}_{A'}\bigr)
  \;=\;
  2\,(q+2)^{3},
\end{equation}
where
\[
  (T_q)_{A'}\;:=\;
  T_{cw,q}^{\boxtimes3}
  \bigl|
  _{A'^{*}\otimes B^{*\otimes3}\otimes C^{*\otimes3}}.
\]

\textbf{Induction setup.}
We proceed by induction on $q$.
When $q = 5$, the assertion follows from a direct computation.
Assume $q\ge5$ and
\[
  \mathrm{rank}\bigl((T_{q-1})^{\wedge1}_{A'}\bigr)
  \;=\;2\,(q+1)^{3}.
\]
Let $S_q:=T_q-T_{q-1} \quad\text{as in \S\ref{sec:Koszul-n3},}$
so that
\(
  (T_q)^{\wedge1}_{A'}=(T_{q-1})^{\wedge1}_{A'}+(S_q)^{\wedge1}_{A'}.
\)
We have a similar lemma for $m=2$ and will use:
\begin{lem}[Pointwise computation, $m=3$]\label{lem:pointwise-3}
Fix $i\in\{1,2\}$.  Let
\(
  \beta_{j_1j_2j_3}\in B^{*\otimes3}
\)
be the dual basis element.
\begin{enumerate}
\item\label{pt:Tq-1-3a}
      If $0\le j_\ell\le q-1$ for all $\ell$ then
      \[
        (T_{q-1}^{\wedge1})_{A'}(e_i\!\otimes\!\beta_{j_1j_2j_3})
        \;=\;
        W_{j_1}\boxtimes W_{j_2}\boxtimes W_{j_3}\,
        (e_i\!\otimes\!\beta_{j_1j_2j_3}),
        \qquad
        (S_q^{\wedge1})_{A'}(e_i\!\otimes\!\beta_{j_1j_2j_3})=0.
      \]
\item\label{pt:Tq-1-3b}
      If at least one $j_\ell=q$ then
      \[
        (S_q^{\wedge1})_{A'}(e_i\!\otimes\!\beta_{j_1j_2j_3})
        \;=\;
        \boxtimes_{\ell} W_{j_\ell}
        (e_i\!\otimes\!\beta_{j_1j_2j_3}),
        \qquad
        (T_{q-1}^{\wedge1})_{A'}(e_i\!\otimes\!\beta_{j_1j_2j_3})=0.
      \]
\end{enumerate}
\end{lem}

\begin{proof}
Use the explicit definition
\[
  (\,\cdot\,)^{\wedge1}_{A'}:
  A'\otimes B^{*\otimes3}\longrightarrow
  \Lambda^{2}A'\otimes C^{\otimes3},
  \quad
  X\otimes\beta_{J}
  \longmapsto
  \sum_{I,K}T^{IJK}\,
     \beta_{J}(b_{J})\,
     \bigl(\phi_2(a_I)\wedge X\bigr)\otimes c_K,
\]
and observe that $T_{q-1}$ never exhibits an index~$q$
while every summand of $S_q$ contains one.
\end{proof}

Define the domain and codomain subspaces explicitly:
\[
  \begin{aligned}
  U_1 &:= A'\!\otimes\!
         \bigl\langle
           \beta_{j_1j_2j_3}
           \,\bigm|\,
           0\le j_1,j_2,j_3\le q-1
         \bigr\rangle,\\
  U_2 &:= A'\!\otimes\!
         \bigl\langle
           \beta_{j_1j_2j_3}
           \,\bigm|\,
           \text{at least one }j_\ell=q
         \bigr\rangle,\\[4pt]
  V_1 &:= \Lambda^{2}A'\!\otimes\!
          \bigl\langle
            c_{j_1j_2j_3}
            \,\bigm|\,
            0\le j_1,j_2,j_3\le q-1
          \bigr\rangle,\\
  V_2 &:= \Lambda^{2}A'\!\otimes\!
          \bigl\langle
            c_{j_1j_2j_3}
            \,\bigm|\,
            \text{at least one }j_\ell=q
          \bigr\rangle.
  \end{aligned}
\]
Then
\(
  A'\!\otimes\!B^{*\otimes3}=U_1\oplus U_2
\)
and
\(
  \Lambda^{2}A'\!\otimes\!C^{\otimes3}=V_1\oplus V_2,
\)
with $V_1\cap V_2=\{0\}$.

\begin{lem}[Image containment, $m=3$]\label{lem:image-containment-3}
With the subspaces above,
\[
  (T_{q-1}^{\wedge1})_{A'}(U_2)=0,
  \qquad
  (S_q^{\wedge1})_{A'}(U_1)=0.
\]
Consequently
\[
  \operatorname{Im}\!\bigl((T_{q-1})^{\wedge1}_{A'}\bigr)\subseteq V_1,
  \qquad
  \operatorname{Im}\!\bigl((S_q)^{\wedge1}_{A'}\bigr)\subseteq V_2.
\]
\end{lem}

\begin{proof}
Immediate from Lemma~\ref{lem:pointwise-3}.
\end{proof}

Lemma~\ref{lem:image-containment-3} and rank additivity imply
\begin{equation}\label{eq:rank-sum-3}
  \mathrm{rank}\bigl((T_q)^{\wedge1}_{A'}\bigr)
  \;=\;
  \mathrm{rank}\bigl((T_{q-1})^{\wedge1}_{A'}\bigr)
  +\mathrm{rank}\bigl((S_q)^{\wedge1}_{A'}\bigr).
\end{equation}

By the induction hypothesis
\(
  \mathrm{rank}\bigl((T_{q-1})^{\wedge1}_{A'}\bigr)=2\,(q+1)^{3}.
\)
Therefore it remains to establish
\begin{equation}\label{eq:Sq-rank-3}
  \mathrm{rank}\bigl((S_q)^{\wedge1}_{A'}\bigr)
  \;=\;
  2\bigl((q+2)^{3}-(q+1)^{3}\bigr)
  \;=\;
  2\,(3q^{2}+9q+7).
\end{equation}
\smallskip\noindent
\textbf{Explicit image of \(\bigl(S_q\bigr)^{\wedge1}_{A'}\) for \(m=3\).}
Because every summand of
\(S_q=T_{cw,q}^{\boxtimes3}-T_{cw,q-1}^{\boxtimes3}\)
carries the index \(q\) in \emph{at least} one of its three
\(B\)--tensor factors \(W_q\),
we analyse the contraction
\[
  (S_q)^{\wedge1}_{A'}:\;
    A'\!\otimes\!B^{*\otimes3}
    \;\longrightarrow\;
    \Lambda^{2}A'\!\otimes\!C^{\otimes3}
\]
only on those inputs whose \(B^{*\otimes3}\)-component evaluates a
triple containing~\(q\).\\
\indent For a multi–index $\mathbf j=(j_{1},j_{2},j_{3})$ we write
$\beta_{\mathbf j}:=\beta_{j_{1}j_{2}j_{3}}$ and   
$c_{\mathbf j}:=c_{j_{1}j_{2}j_{3}}$. By Lemma \ref{lem:image-containment-3}, we have the list

\subsubsection*{Inputs with $X=e_{1}$ (total $\;3q^{2}+3q+1$)}
\begin{enumerate}[label=(\arabic*),leftmargin=2.3em,start=1]
\item\label{e1:first}%
\textbf{Family $\mathcal A_{1}$:}
$e_{1}\otimes\beta_{q i j}$ with $0\le i,j\le q$
\hfill(\# $(q+1)^{2}$)\\
\hspace*{1em}%
$(S_{q})^{\wedge1}_{A'}(e_{1}\!\otimes\!\beta_{q i j})
      =e_{0}\wedge e_{1}\otimes c_{q i j}
      \;+\;\delta_{i,2}\,e_{2}\wedge e_{1}\otimes c_{q0j}$.
\item\label{e1:second}%
\textbf{Family $\mathcal A_{2}$:}
$e_{1}\otimes\beta_{i q j}$ with $0\le i\le q-1,\;0\le j\le q$
\hfill(\# $q(q+1)$)\\
\hspace*{1em}%
$(S_{q})^{\wedge1}_{A'}(e_{1}\!\otimes\!\beta_{i q j})
      =e_{0}\wedge e_{1}\otimes c_{i q j}
       \;+\;\delta_{i,2}\,e_{2}\wedge e_{1}\otimes c_{0 q j}$.
\item\label{e1:third}%
\textbf{Family $\mathcal A_{3}$:}
$e_{1}\otimes\beta_{i j q}$ with $0\le i,j\le q-1$
\hfill(\# $q^{2}$)\\
\hspace*{1em}%
$(S_{q})^{\wedge1}_{A'}(e_{1}\!\otimes\!\beta_{i j q})
      =e_{0}\wedge e_{1}\otimes c_{i j q}
       \;+\;\delta_{i,2}\,e_{2}\wedge e_{1}\otimes c_{0 j q}$.
\end{enumerate}

\subsubsection*{Inputs with $X=e_{2}$ (total $\;3q^{2}+3q+1$)}
\begin{enumerate}[label=(\arabic*),leftmargin=2.3em,start=4]
\item\label{e2:first}%
\textbf{Family $\mathcal B_{1}$:}
$e_{2}\otimes\beta_{q i j}$ with $0\le i,j\le q$
\hfill(\# $(q+1)^{2}$)\\
\hspace*{1em}%
$(S_{q})^{\wedge1}_{A'}(e_{2}\!\otimes\!\beta_{q i j})
      =e_{0}\wedge e_{2}\otimes c_{q i j}
      \;+\;\delta_{i,1}\,e_{1}\wedge e_{2}\otimes c_{q0j}$.
\item\label{e2:second}%
\textbf{Family $\mathcal B_{2}$:}
$e_{2}\otimes\beta_{i q j}$ with $0\le i\le q-1,\;0\le j\le q$
\hfill(\# $q(q+1)$)\\
\hspace*{1em}%
$(S_{q})^{\wedge1}_{A'}(e_{2}\!\otimes\!\beta_{i q j})
      =e_{0}\wedge e_{2}\otimes c_{i q j}
       \;+\;\delta_{i,1}\,e_{1}\wedge e_{2}\otimes c_{0 q j}$.
\item\label{e2:third}%
\textbf{Family $\mathcal B_{3}$:}
$e_{2}\otimes\beta_{i j q}$ with $0\le i,j\le q-1$
\hfill(\# $q^{2}$)\\
\hspace*{1em}%
$(S_{q})^{\wedge1}_{A'}(e_{2}\!\otimes\!\beta_{i j q})
      =e_{0}\wedge e_{2}\otimes c_{i j q}
       \;+\;\delta_{i,1}\,e_{1}\wedge e_{2}\otimes c_{0 j q}$.
\end{enumerate}

\subsubsection*{Inputs with $X=e_{0}$ (total $\;12(q+1)$)}
Here exactly \emph{one} index of $\mathbf j$ is $q$ and
\emph{one} index is $1$ or $2$; the remaining free index runs over
$0,\dots,q$.  Set
\[
\mathcal C_{t}^{(\ell)}
  :=\{\mathbf j=(j_{1},j_{2},j_{3})\mid
       j_{\ell}=q,\;
       j_{m}=t,\;
       j_{n}=r,\;
       0\le r\le q\},
\quad
t\in\{1,2\},\; \ell\in\{1,2,3\},
\]
where \((\ell,m,n)\) is a cyclic permutation of $(1,2,3)$.  All twelve
sets \(\mathcal C_{t}^{(\ell)}\) are disjoint, each of size $(q+1)$,
so together they contribute \(12(q+1)\) inputs.

\begin{enumerate}[label=(\arabic*),leftmargin=2.3em,start=7]
\item\label{e0:general}%
\textbf{Family $\mathcal C_{t}^{(\ell)}$:}
$e_{0}\otimes\beta_{\mathbf j}$ with $\mathbf j\in\mathcal C_{t}^{(\ell)}$
\hfill(\# $(q+1)$ in each of the 12 families)\\
\hspace*{1em}%
Let $t\in\{1,2\}$ and suppose $j_{m}=t$. 
Write $\mathbf j^{\langle0\rangle}$ for the multi–index obtained from
$\mathbf j$ by replacing that $t$ by $0$ (all other coordinates unchanged).
Then
\[
\bigl(S_{q}\bigr)^{\wedge1}_{A'}
   \bigl(e_{0}\otimes\beta_{\mathbf j}\bigr)
   \;=\;
   e_{t}\wedge e_{0}\;\otimes\;c_{\mathbf j^{\langle0\rangle}}.
\]
\end{enumerate}

\medskip
\textbf{Conclusion}
Putting the three blocks together one obtains
\[
\#\operatorname{Im}\!\bigl((S_{q})^{\wedge1}_{A'}\bigr)
  \;=\;
  2\bigl(3q^{2}+3q+1\bigr)\;+\;12(q+1)
  \;=\;
  6q^{2}+18q+14,
\]
and the displayed images are visibly linearly independent because
\(\Lambda^{2}A'\) carries the disjoint basis
\(\{e_{0}\wedge e_{1},\,e_{0}\wedge e_{2},\,e_{1}\wedge e_{2}\}\)
and every term listed above features a \emph{unique} wedge factor and
a \emph{unique} $C$–tensor an in the case of $m=2$.
Consequently  
\(\mathrm{rank}\bigl((S_{q})^{\wedge1}_{A'}\bigr)=6q^{2}+18q+14\),
and complete the proof by \eqref{eq:Sq-rank-3}.
\end{proof}
Although experts may already know the answer, the author does not know how to express the output matrix of the CW matrix-multiplication algorithm solely in terms of the entries of the input matrices.

\section*{Acknowledgements}
This note originated from questions that arose during the author's work on a ciphertext-matrix-multiplication project at EAGLYS Inc. The author is grateful to his colleagues there for providing a stimulating environment and for many helpful discussions.

\end{document}